\documentclass[reqno]{amsart}

\usepackage{amsmath,amssymb,mathbbol}

\usepackage[active]{srcltx}
\usepackage{graphicx,color}

\usepackage{epic}
\usepackage{pstricks}
\usepackage{amsthm}

\newtheorem{theorem}{Theorem}[section]

\newtheorem{lemma}[theorem]{Lemma}

\newtheorem{corollary}[theorem]{Corollary}
\theoremstyle{definition}

\newtheorem{examples}[theorem]{Examples}
\newtheorem{remark}[theorem]{Remark}

\newtheorem{remarks}[theorem]{Remarks}

\def\bT{\mathbb{T}}

\def\R{\mathbb{R}}

\def\N{\mathbb{N}}

\def\E{\mathcal{E}}

\def\cC{{\mathcal C}}

\def\cL{{\mathcal L}}
\def\cM{{\mathcal M}}

\def\Lip{{\mathcal Lip}}
\def\GP{{\mathcal GP}}

\newcommand{\dom}[1]{\mathsf{D}_{#1}}

\def\trXY1-tq{Tr(X,Y, 1-\theta ,q)}

\definecolor{darkred}{rgb}{0.7,0.1,0.1}

\newcommand{\vanish}[1]{\relax}

\title[Interpolation of nonlinear order preserving operators]{Interpolation of nonlinear positive or order preserving operators on Banach lattices}

\author{Ralph Chill}

\address{R.~Chill, Institut f\"ur Analysis, Fakult\"at f\"ur Mathematik, TU Dresden, 01062 Dresden, Germany}
\email{ralph.chill@tu-dresden.de}

\author{Alberto Fiorenza}
\address{A.~Fiorenza, Dipartimento di Architettura - 
Universit\'a di Napoli, Via Monteoliveto, 3,
I-80134 Napoli, Italy\\
and Istituto per le Applicazioni del Calcolo
"Mauro Picone", sezione di Napoli \\
Consiglio Nazionale delle Ricerche \\
via Pietro Castellino, 111 \\
 I-80131 Napoli, Italy }
\email{fiorenza@unina.it} 

\author{Sebastian Kr\'ol}
\address{S. Kr\'ol, Faculty of Mathematics and Computer Science, Nicolaus Copernicus University, ul. Chopina 12/18, 87-100 Toru\'n, Poland}
\email{sebastian.krol@mat.umk.pl}

\thanks{The first author is grateful for support by the Deutsche Forschungsgemeinschaft (grant CH 1285/5-1, Order preserving operators in problems of optimal control and in the theory of partial differential equations). The third author was partially supported by the Alexander von Humboldt Foundation and Narodowe Centrum Nauki grant DEC-2011/03/B/ST1/00407}

\begin{document}

\date{\today}



\begin{abstract} 
We study the relationship between exact interpolation spaces for positive, linear operators, for order preserving, Lipschitz continuous operators, and for positive Gagliardo-Peetre operators, and exact partially $K$-monotone spaces in interpolation couples of compatible Banach lattices. By general Banach lattice theory we recover a characterisation of exact interpolation spaces for order preserving, Lipschitz continuous operators in the couple $(L^1,L^\infty )$ due to B\'enilan and Crandall. 
\end{abstract}

\renewcommand{\subjclassname}{\textup{2010} Mathematics Subject Classification}

\maketitle

\section{Introduction}
 
Motivated by the study of nonlinear semigroups acting on the whole scale of $L^p$ spaces, B\'enilan and Crandall in their article \cite{BeCr91} on completely accretive operators introduced the class of {\em normal spaces}. These spaces can be described as the class of intermediate spaces $X$ between $L^1$ and $L^\infty$ such that every {\em order preserving}, possibly nonlinear operator $S:L^1 + L^\infty \supseteq \dom{S}\to L^1 + L^\infty$ which is contractive both with respect to the norms in $L^1$ and $L^\infty$, and whose domain $\dom{S}$ satisfies an additional compatibility condition with respect to the order structure, is contractive also with respect to the norm in $X$. In other words, normal spaces are the exact interpolation spaces for admissible, order preserving, Lipschitz continuous operators on the interpolation couple $(L^1,L^\infty )$. 

In this note we study the problem of understanding exact interpolation spaces for order preserving, Lipschitz continuous operators (and other operators) in general interpolation couples of Banach lattices. An important ingredient in the approach by B\'enilan and Crandall, which is a variant of the Riesz interpolation theorem due to Brezis and Strauss \cite{BrSt73}, and which uses basic properties of integrals, is here not available. Any proof can only rely on the general theory of Banach lattices. For us, a particular emphasis lies in the identification of {\em exact} interpolation spaces for three classes of operators, namely for positive, linear operators, for order preserving, Lipschitz continuous operators, and for positive Gagliardo-Peetre operators (see Section 4 for the definition of these operators). ``Exact'' interpolation space means here that operator norms of admissible operators in the interpolation space are less than or equal to the respective operator norms in the endpoint spaces of the 
interpolation couple. Besides the various exact interpolation spaces we further introduce the notion of exact partially $K$-monotone spaces; see Section 2. As a corollary to our abstract results we obtain the following theorem which extends the result by B\'enilan and Crandall mentioned above: it says that in the case of the interpolation couple $(L^1,L^\infty )$ over a $\sigma$-finite measure space, the various classes of exact interpolation spaces, the class of exact partially $K$-monotone spaces and the class of normal spaces coincide. 

\begin{theorem} \label{thm.main.l1.linfty}
 Let $(\Omega , \mu)$ be a $\sigma$-finite measure space, and let $X$ be an intermediate space for the interpolation couple $(L^1 (\Omega ),L^\infty (\Omega ))$. Then the following are equivalent:
 \begin{itemize}
  \item[(i)]  The space $X$ is an exact partially monotone space ($=$ normal space in the terminology of B\'enilan and Crandall).
  \item[(ii)] The space $X$ is an exact partially $K$-monotone space. 
  \item[(iii)] The space $X$ is an exact interpolation space for positive Gagliardo-Peetre operators. 
  \item[(iv)] The space $X$ is an exact interpolation space for (admissible) order preserving, Lipschitz continuous operators with solid lattice domain.
  \item[(v)] The space $X$ is an exact interpolation space for (admissible) positive, linear operators.
 \end{itemize}
\end{theorem}

Our abstract results actually say that the above theorem remains true in a more general class of interpolation couples of compatible Banach lattices in which a variant of the Calder\'on-Ryff theorem remains ture, and that the implications (ii)$\Leftrightarrow$(iii)$\Rightarrow$(iv)$\Rightarrow$(v) are true in all interpolation couples of compatible Banach lattices; see Theorem \ref{thm.gp}. 

We emphasize that the question of being an {\em exact} interpolation space is also a question of the underlying norm of the space. We prove that the interpolation spaces for positive, linear operators and the interpolation spaces for bounded, linear operators (both not necessarily exact) coincide for all interpolation couples of compatible Banach lattices; see Lemma 2.2.  This may be somewhat surprising since not every admissible operator on the endpoint spaces need be regular, that is, it need not be the difference of two positive operators. And in fact, as Lozanovskii's example shows it does not hold for general interpolation couples of Banach lattices; cf. \cite{Lo72}, \cite{Ber81}, and \cite{Sh81}. 
We point out that the fact that some interpolation spaces for positive, linear operators and the interpolation spaces for bounded, linear operators coincide is known to hold in certain interpolation couples of ordered Banach spaces, too, namely in the case of noncommutative $L^1$ and $L^\infty$ spaces; see Veselova, Sukochev and Tikhonov \cite[Theorem 5.1]{VeSuTi07}. Concerning {\em exact} interpolation spaces, however, we give easy examples of interpolation spaces of interpolation couples of Banach lattices which are exact for positive, linear operators but which are not exact for bounded, linear operators; see Remarks \ref{examples} and Lemma \ref{lem.wishful}(d).    

This note is of course related to classical topics in interpolation theory, for example to the identification of exact interpolation spaces for bounded, linear operators; we refer to the monographs by Bergh and L\"ofstr\"om \cite{BeLo76}, Bennett and Sharpley \cite{BeSh88} or Brudnyi and Krugljak \cite{BrKr91}. It is also related to the topic of interpolation of Lipschitz continuous operators which goes back to Orlicz \cite{Or54}, Lorentz and Shimogaki \cite{LoSh68}, Browder \cite{Br69} and others. We mention that there are important, recent extensions of the theory, going beyond the case of Lipschitz continuous operators, due to Coulhon and Hauer \cite[Section 4]{CoHa18}. To the best of our knowledge, it seems however that the interpolation of order preserving or positive operators is much less studied and that a characterisation of the exact interpolation spaces is missing in the context of Banach lattices. We hope to make a step towards a better understanding of this part of interpolation theory.

\section{Partially $K$-monotone intermediate spaces of interpolation couples of Banach lattices}

We recall a few known notions from interpolation theory. Let $\vec{X} = (X_0,X_1)$ be an interpolation couple of Banach spaces. We equip the spaces $X_0\cap X_1$ and $X_0 +X_1$ with the respective norms
\begin{align*}
 \| f\|_{X_0\cap X_1} & := \| f\|_{X_0} + \| f\|_{X_1} \quad \text{and} \\
 \| f\|_{X_0+X_1} & := \inf \{ \| f_0 \|_{X_0} + \| f_1\|_{X_1} : f_i\in X_i \text{ and } f=f_0+f_1 \} ,
\end{align*}
so that both spaces are Banach spaces. The $K$-functional from interpolation theory is given by
 \[
  K(f,t) = \inf \{ \| f_0\|_{X_0} + t\, \| f_1\|_{X_1} : f=f_0 + f_1\} 
 \]
for $f\in X_0+X_1$ and $t\in (0,\infty )$. With the help of the $K$-functional we introduce the relation $\preceq_K$ on $X_0+X_1$ by setting 
\[
g\preceq_K f \quad : \Leftrightarrow \quad K(g,\cdot )\leq K(f,\cdot ) \text{ pointwise everywhere on } (0,\infty ). 
\]
Note that $\preceq_K$ is not an order relation, since $g\preceq_K f$ and $f\preceq_K g$ does not imply $g=f$. An {\em intermediate space} of the couple $\vec{X}$ is a Banach space $X$ such that $X_0\cap X_1 \subseteq X \subseteq X_0+X_1$ with continuous embeddings. An intermediate space $X$ is {\em $K$-monotone} if there exists a constant $C\geq 1$ such that, for every $f\in X$ and every $g\in X_0+X_1$, 
\[
 g\preceq_K f \quad \Rightarrow \quad g\in X \text{ and } \| g\|_X \leq C\, \| f\|_X ,
\]
and it is {\em exactly $K$-monotone} if this implication holds true with $C=1$.

We say that an interpolation couple $\vec{X} = (X_0,X_1)$ of Banach spaces is an interpolation couple of compatible Banach lattices if $X_0$, $X_1$ and $X_0+X_1$ are Banach lattices, and if $X_0$ and $X_1$ are order ideals in $X_0+X_1$. The following auxiliary lemma contains some technical observations on decompositions of elements in Banach lattices and on the $K$-functional on interpolation couples of compatible Banach lattices. 

\begin{lemma} \label{lem.basic}
 Let $\vec{X} = (X_0,X_1)$ be an interpolation couple of compatible Banach lattices. Then: 
\begin{itemize}
 \item[(a)] For every $f\in X_0+X_1$ and every $f_i\in X_i$ with $f=f_0+f_1$ there exist $\hat{f}_i\in X_i$ such that 
 \[
 0\leq \hat{f}_i^\pm \leq f_i^\pm \text{ and } f^\pm = \hat{f}_0^\pm +\hat{f}_1^\pm . 
 \]
 In particular, $f= \hat{f}_0+\hat{f}_1$ and $|f| = |\hat{f}_0| + |\hat{f}_1|$. 
 \item[(b)] For every $f\in X_0+X_1$ and every $g_i\in X_i^+$ with $|f| = g_0+g_1$ there exists $f_i\in X_i$ such that $f= f_0+f_1$ and $|f_i| = g_i$.
 \item[(c)] For every $f\in X_0+X_1$ and every $t\in (0,\infty )$,
\[
 K(f,t) = \inf \{ \| g_0\|_{X_0} + t \, \| g_1\|_{X_1} : g_i\in X_i^+ \text{ and } |f| = g_0 +g_1 \} = K(|f|,t) .
\]
In particular, if $f$, $g\in X_0+X_1$ and $|g|\leq |f|$, then $K(g,\cdot ) \leq K(f,\cdot )$ pointwise everywhere on $(0,\infty )$, that is, $g\preceq_K f$. 
\end{itemize}
\end{lemma}

\begin{proof}
\noindent (a) Let $f\in X_0+X_1$ and $f_i\in X_i$ such that $f=f_0+f_1$. Let $h^+ = f_0^+ \wedge f_1^-$ and $h^- = f_0^- \wedge f_1^+$. Since $X_0$ and $X_1$ are ideals in $X_0+X_1$, then $h^+$, $h^-\in X_0\cap X_1$, and clearly $h^+$, $h^- \geq 0$ and $h^+ \wedge h^- = 0$. Let $h=h^+-h^-$, and set $\hat{f}_0 = f_0 -h \in X_0$ and $\hat{f}_1 := f_1+h\in X_1$. Then $f = \hat{f}_0 + \hat{f}_1$. Note that 
\begin{align*}
 & 0\leq f_0^+ - h^+ \leq f_0^+ , \\
 & 0\leq f_0^- - h^- \leq f_0^- , \\
 & 0\leq f_1^+ - h^- \leq f_1^+ , \text{ and} \\
 & 0\leq f_1^- - h^+ \leq f_1^- .
\end{align*}
Since by \cite[Theorem 1.1.1 (iv), p. 3]{MN91} the representation of an element in a Banach lattice as a difference of two disjoint, positive elements is unique, $\hat{f}_0^+ = f_0^+ - h^+$, $\hat{f}_0^- = f_0^- -h^-$, $\hat{f}_1^+ = f_1^+ -h^-$ and $\hat{f}_1^- = f_1^- - h^+$. Moreover, $\hat{f}_0^+ \wedge \hat{f}_1^- = 0$ and $\hat{f}_0^- \wedge \hat{f}_1^+ = 0$. Since, in addition, $\hat{f}_0^+ \wedge \hat{f}_0^- = 0$ and $\hat{f}_1^+ \wedge \hat{f}_1^- = 0$, we deduce $\hat{f}_0^+ \wedge (\hat{f}_0^- + \hat{f}_1^-) = 0$ and $\hat{f}_1^+ \wedge (\hat{f}_0^- + \hat{f}_1^-) = 0$, and therefore $(\hat{f}_0^+ +\hat{f}_1^+) \wedge (\hat{f}_0^- + \hat{f}_1^-) = 0$. Hence, by \cite[Theorem 1.1.1 (iv), p. 3]{MN91} again,
\begin{align*}
 f^+ & = (\hat{f}_0 + \hat{f}_1 )^+  = \hat{f}_0^+ + \hat{f}_1^+ ,
\end{align*}
and similarly $f^- = \hat{f}_0^- + \hat{f}_1^-$. We have proved the first part of the statement. The rest is a straightforward consequence of this first part.

\noindent (b) Set $f_i = g_i \wedge f^+ - g_i\wedge f^-$. 

\noindent (c) Let $f\in X_0+X_1$ and $t>0$. On the one hand, by definition of the $K$-functional, since $X_0$ and $X_1$ are Banach lattices, and by properties of the infimum, 
\begin{align*}
 K(f,t) & = \inf \{ \| f_0\|_{X_0} + t \| f_1\|_{X_1} : f_i\in X_i \text{ and } f=f_0 + f_1\} && \\
 & = \inf \{ \| \, |f_0|\, \|_{X_0} + t \| \, |f_1|\, \|_{X_1} : f_i\in X_i \text{ and } f=f_0 + f_1\} && \\
 & \geq  \inf \{ \| g_0\|_{X_0} + t \, \| g_1\|_{X_1} : g_i\in X_i^+ \text{ and } |f| = g_0 +g_1 \} && \text{(by part (a))} \\
 & \geq  \inf \{ \| g_0\|_{X_0} + t \, \| g_1\|_{X_1} : g_i\in X_i \text{ and } |f| = g_0 +g_1 \} && \\
 & = K(|f|,t). &&
\end{align*}
On the other hand,  
\begin{align*}
K(|f|,t) & = \inf \{ \| g_0\|_{X_0} + t \, \| g_1\|_{X_1} : g_i\in X_i^+ \text{ and } |f| = g_0 +g_1 \} && \text{(by part (a))} \\
  & \geq \inf \{ \| f_0\|_{X_0} + t \| f_1\|_{X_1} : f_i\in X_i \text{ and } f=f_0 + f_1\} && \text{(by part (b))} \\
  & = K(f,t) .
\end{align*}
\end{proof}

In an interpolation couple $\vec{X} = (X_0,X_1)$ of compatible Banach lattices, we introduce another relation on $X_0+X_1$, denoted by $\ll_K$, by setting
\begin{align*}
g \ll_K f \quad & : \Leftrightarrow  \quad K(g^\pm , \cdot ) \leq K (f^\pm , \cdot ) \text{ pointwise everywhere on } (0,\infty ) \\
& \phantom{:} \Leftrightarrow \quad g^+ \preceq_K f^+ \text{ and } g^- \preceq_K f^- . 
\end{align*}
It is important to note that the two relations $\preceq_K$ and $\ll_K$ coincide on the positive cone of $X_0+X_1$, but they are in general not equal. 
An intermediate space $X$ is {\em partially $K$-monotone} if there exists a constant $C\geq 1$ such that, for every $f\in X$ and every $g\in X_0+X_1$ 
\[
  g\ll_K f \quad \Rightarrow \quad g\in X \text{ and } \| g\|_X \leq C\, \| f\|_X ,
\]
and it is {\em exactly partially $K$-monotone} if this implication holds true with $C=1$.

\begin{lemma} \label{lem.wishful}
Let $\vec{X} = (X_0,X_1)$ be an interpolation couple of compatible Banach lattices, and let $X$ be an intermediate space. Then:
\begin{itemize}
\item[(a)] For every $f$, $g\in X_0+X_1$, $g\ll_K f$ implies $g\preceq_K 2f$. In particular, if $X$ is $K$-monotone, then $X$ is partially $K$-monotone. 
\item[(b)] If $X$ is an exact $K$-monotone space, then $X$ is a Banach lattice and an order ideal in $X_0+X_1$. 
\item[(c)] If $X$ is an exact partially $K$-monotone space, then $X$ is a vector lattice and an order ideal in $X_0+X_1$, and for every $f\in X$ 
\begin{equation} \label{eq.norm.equivalence}
 \frac12 \,\| \, |f|\, \|_X  \leq \| f\|_X \leq 2\, \| \, |f|\, \|_X .
\end{equation}
The norm $\| f\|_* := \| \, |f|\, \|_X$ ($f\in X$) is an equivalent norm for which $X$ becomes exactly $K$-monotone. 
\item[(d)] If $X$ is an exact partially $K$-monotone space, then the following are equivalent:
\begin{itemize}
\item[(i)] The space $X$ is an exact $K$-monotone space. 
\item[(ii)] The space $X$ is a Banach lattice.
\item[(iii)] For every $f\in X$ one has $\| f\|_X = \|\, |f|\, \|_X$. 
\end{itemize}
\end{itemize}
\end{lemma}

\begin{proof}
(a) Let $f$, $g\in X_0+X_1$ be such that $g\ll_K f$, that is $K(g^\pm ,\cdot ) \leq K(f^\pm ,\cdot )$ pointwise everywhere. Then, by Lemma \ref{lem.basic} (c), for every $t\in (0,\infty )$,
\begin{align*}
 K(g,t) & \leq K(g^+ ,t) + K(g^-,t) \\
  & \leq K(f^+,t) + K(f^-,t) \\
  & \leq 2\, K(|f|,t) \\
  & = 2\, K(f,t) , 
\end{align*}
which means $g\preceq_K 2\, f$. From here follows directly the implication that $K$-monotonicity implies partial $K$-monotonicity.

(b) Assume that $X$ is exactly $K$-monotone. Let $f\in X$ and $g\in X_0+X_1$ be such that $|g|\leq |f|$. Then, by Lemma \ref{lem.basic} (c), 
\[
 K(g,\cdot ) = K(|g|,\cdot ) \leq K(|f|,\cdot ) = K(f,\cdot ) , 
\]
so that
\[
 g\preceq_K |g| \preceq_K g \preceq_K f \preceq_K |f| \preceq_K f .
\]
Since $X$ is exactly $K$-monotone, this implies $g$, $|g|$, $|f|\in X$ and 
\[
 \| g\|_X = \| \, |g| \,\|_X \leq \| \, |f|\, \|_X = \| f\|_X .
\]
As a consequence, $X$ is an order ideal in $X_0+X_1$ and a Banach lattice.

 (c) Assume that $X$ is an exact partially $K$-monotone space. By Lemma \ref{lem.basic} (c) again, for every $f\in X_0+X_1$,
\[
 f^+ , \, -f^- \ll_K f \quad \text{ and } \quad f^+ ,\, f^- \ll_K |f| .
\]
Hence, for every $f\in X$ we deduce $f^\pm\in X$ and thus $|f|\in X$ (so that $X$ is a vector lattice), and 
\[
\| f^\pm \|_X \leq \| f\|_X \quad \text{ and } \quad \| f^\pm \|_X \leq \| \, |f|\, \|_X .
\]
As a consequence, 
\begin{align*}
 & \| \, |f|\, \|_X = \| f^+ + f^-\|_X \leq \| f^+ \|_X + \| f^-\|_X \leq 2\, \| f\|_X \quad \text{and} \\
 & \| f \|_X = \| f^+ - f^-\|_X \leq \| f^+ \|_X + \| f^-\|_X \leq 2\, \| \, |f|\, \|_X ,
\end{align*}
which implies the inequalities in \eqref{eq.norm.equivalence}. As a consequence, $\|\cdot \|_*$ is an equivalent norm on $X$, and using that the relations $\preceq_K$ and $\ll_K$ coincide on the positive cone of $X_0+X_1$ it is straightforward to show that $(X,\|\cdot \|_*)$ is an exact $K$-monotone space. 

(d) The implication (i)$\Rightarrow$(ii) follows from assertion (b) above, and the implication (ii)$\Rightarrow$(iii) is true in every Banach lattice. In order to show the implication (iii)$\Rightarrow$(i), assume that $X$ is an exact partially $K$-monotone space and that for every $f\in X$ one has $\| f\|_X = \|\, |f|\,\|_X$. Let $f\in X$ and $g\in X_0+X_1$ be such that $g\preceq_K f$. This relation and Lemma \ref{lem.basic} (c) imply $K(|g|,\cdot ) = K(g,\cdot ) \leq K(f,\cdot ) = K(|f|,\cdot )$, that is, $|g|\preceq_K |f|$. Since for positive elements the relations $\preceq_K$ and $\ll_K$ coincide, and since $X$ is partially $K$-monotone, we deduce $|g|\in X$ and $\| \, |g|\, \|_X \leq \| \, |f|\, \|_X$. In a similar way one shows $g^\pm\in X$, and hence $g\in X$. The norm equality (assumption (iii)) finally yields $\| g\|_X \leq \| f\|_X$. Hence, $X$ is exactly $K$-monotone.  
\end{proof}

\begin{remarks}\label{examples}
(a) In general, the constants $\frac12$ and $2$ in \eqref{eq.norm.equivalence} can not be globally improved. In fact, the space $X = \R^2$ equipped with either of the norms
\[
  N_1 (x_1,x_2) := \begin{cases}
                             |x_1| + |x_2| & \text{if } x_1 x_2 \geq 0 , \\[2mm]
                             \sup \{ |x_1| , |x_2| \} & \text{if } x_1 x_2 < 0 ,
                            \end{cases}
\]
or 
\[
  N_2 (x_1,x_2) := \begin{cases}
                             \sup \{ |x_1| , |x_2| \} & \text{if } x_1 x_2 \geq 0 , \\[2mm]
                             |x_1| + |x_2| & \text{if } x_1 x_2 < 0 ,
                            \end{cases}
\]
(for $(x_1,x_2)\in\R^2$) is an exact partially $K$-monotone intermediate space between $(\R^2 , \|\cdot \|_1 )$ and $(\R^2 , \|\cdot \|_\infty )$. In $(\R^2 , N_1)$ the constant $\frac12$ in the first inequality in \eqref{eq.norm.equivalence} is optimal, while in $(\R^2 , N_2 )$ the constant $2$ in the second inequality in \eqref{eq.norm.equivalence} is optimal; check for example with the vector $(x_1,x_2) = (1,-1)$. Since $N_i ((1,-1)) \not= N_i (|(1,-1)|) = N_i ((1,1))$, the spaces $(\R^2 , N_i)$ are not Banach lattices, and therefore, by Lemma \ref{lem.wishful} (d), the spaces $(\R^2,N_i)$ are not exactly $K$-monotone.  
\begin{figure}[t]
  \centering
  \includegraphics[width=6.2cm]{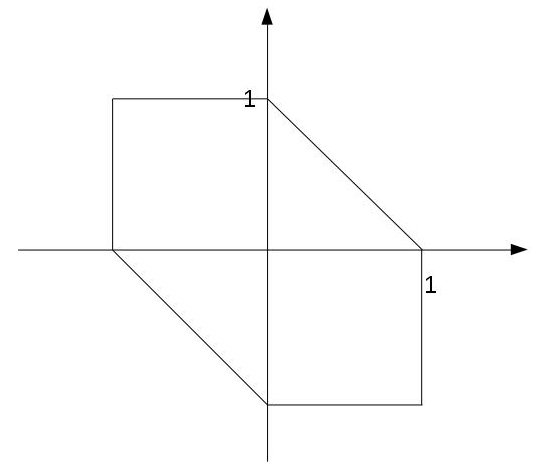}\ 
  \includegraphics[width=6.2cm]{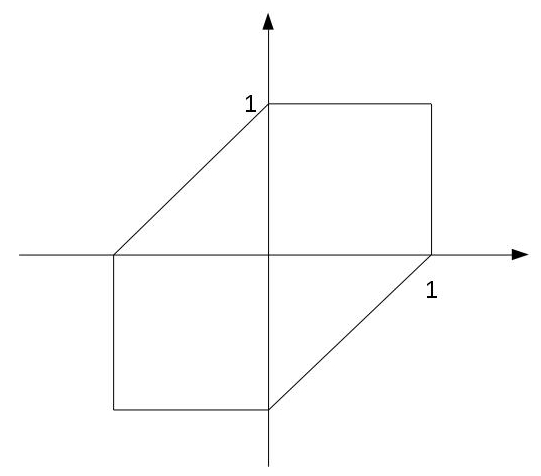} \\
  \caption{Unit balls of the norms $N_1$ (left) and $N_2$ (right) on $\R^2$}
  \label{fig.norms}
\end{figure}

(b) Similarly, the spaces $X=L^1(\Omega )$ with the norm $\|f\|_X:=\max \{ \|f^+\|_{L_1}, \|f^-\|_{L_1} \}$ and $X=L^\infty(\Omega )$ with $\|f\|_{X}:=\|f^+\|_{L^\infty}+\|f^-\|_{L^\infty}$ are exact partially $K$-monotone intermediate spaces of the couple $(L^1(\Omega ), L^\infty (\Omega ))$, and for these spaces the constant $\frac12$ and $2$ in the inequality \eqref{eq.norm.equivalence} are optimal, too.
\end{remarks}

\section{Interpolation of positive, linear operators}

Let $\vec{X} = (X_0,X_1)$ and $\vec{Y} = (Y_0,Y_1)$ be two interpolation couples of Banach spaces. A linear operator $S:X_0+X_1\to Y_0+Y_1$ is an {\em admissible, bounded, linear operator} if $SX_i \subseteq Y_i$ and if the restrictions $S: X_i \to Y_i$ are bounded ($i=0$, $1$). We denote the space of all admissible, bounded, linear operators between both couples by $\cL (\vec{X} , \vec{Y})$; we shortly write $\cL (\vec{X} ) := \cL (\vec{X} , \vec{X} )$. The space $\cL (\vec{X} , \vec{Y})$ is a Banach space when equipped with the norm
\[
 \| S\|_{\cL (\vec{X},\vec{Y})} := \sup \{ \| S\|_{\cL (X_i,Y_i)} : i=0, 1 \} .
\]
An intermediate space $X\subseteq X_0+X_1$ is called an {\em interpolation space} (more precisely, {\em interpolation space for bounded, linear operators}) if every admissible, bounded, linear operator $S\in\cL (\vec{X})$ leaves $X$ invariant, that is, $SX\subseteq X$. By the closed graph theorem, $S$ is then necessarily bounded on $X$. An interpolation space $X$ is {\em exact} if for every admissible, bounded, linear operator $S\in\cL (\vec{X})$
\begin{equation} \label{eq.exact}
 \| S\|_{\cL (X)} \leq \| S\|_{\cL (\vec{X})} .
\end{equation}

If $\vec{X} = (X_0,X_1)$ and $\vec{Y} = (Y_0,Y_1)$ are two interpolation couples of compatible Banach lattices, then we denote by $\cL^+ (\vec{X},\vec{Y})$ the cone of all admissible, {\em positive}, linear operators; recall from \cite[Proposition 1.3.5, p. 27]{MN91} that every positive, linear operator on a Banach lattice is automatically bounded, so that $\cL^+ (\vec{X},\vec{Y})\subseteq \cL (\vec{X},\vec{Y})$. We call an intermediate space $X$ an {\em interpolation space for positive, linear operators} if every admissible, {\em positive}, linear operator $S\in \cL^+ (\vec{X})$ leaves $X$ invariant, and we call it {\em exact} if in addition for these operators the inequality \eqref{eq.exact} holds. Since there are fewer {\em positive}, linear operators than bounded, linear operators, the classes thus defined are a priori larger than the classes of all (exact) interpolation spaces; cf. \cite{Lo72}, \cite{Ber81}, and \cite{Sh81}. If $X$ is an interpolation space for positive, linear operators, then 
\[
 \| f\|_* := \sup \{ \| Tf \|_X : T \in \cL^+ (\vec{X}), \, \| T\|_{\cL (X)} \leq 1\}
\]
defines an equivalent norm for which $X$ becomes an exact interpolation space for positive, linear operators; compare with \cite[Proposition 1.15, p. 105]{BeSh88}, which is the corresponding result for interpolation spaces for bounded, linear operators.  \\

We call an interpolation couple $\vec X = (X_0,X_1)$ of compatible Banach lattices an {\em exact Calder\'on-Mityagin couple  with respect to positive operators} or an {\em exact partially $K$-monotone couple with respect to positive operators}, if for every {\em positive} $f$, $g\in X_0+X_1$ with $K(g,\cdot )\leq K(f,\cdot )$ there exists an admissible, {\em positive}, linear operator $T\in\cL^+ (\vec{X})$ such that $\|T\|_{\cL (\vec{X})}\leq 1$ and $Tf =g$. \\

The following theorem shows the relation between exact partially $K$-monotone intermediate spaces and exact interpolation spaces for positive, linear operators. Before, however, we state an auxiliary lemma. 

\begin{lemma} \label{lem.dedekind}
 Let $\vec{X} = (X_0,X_1)$ be an interpolation couple of compatible Banach lattices. If $X_0$ and $X_1$ are $\sigma$-Dedekind complete, then $X_0+X_1$ is $\sigma$-Dedekind complete. 
\end{lemma}

\begin{proof}
 Let $(h^n)$ be a sequence in $X_0+X_1$ which is bounded from above by $g\in X_0+X_1$. Replacing $h^n$ by $(h^1 \vee \dots \vee h^n) -h^1$, we may without loss of generality assume that the sequence is positive and increasing. We write $g= g_0 + g_1$ with $g_i\in X_i^+$, and we set $h_0^n := h^n \wedge g_0 = h^n - (h^n -g_0)^+$ and $h_1^n := h^n - h_0^n = (h^n-g_0)^+$. Then the sequences $(h_0^n)$ and $(h_1^n)$ are increasing and, for every $n$, 
 \begin{align*}
  h_0^n & \leq g_0 \text{ and} \\
  h_1^n & \leq (g_0+g_1 -g_0)^+ = g_1 .
 \end{align*}
In particular, since the spaces $X_i$ are order ideals in $X_0+X_1$, $h_i^n \in X_i$ for every $n$. Since the spaces $X_i$ are $\sigma$-Dedekind complete, there exists $h_i := \sup_n h_i^n$ in $X_i$. Moreover, by \cite[Proposition 1.1.10]{MN91}, 
\[
 h_i = o-\lim_{n\to\infty} h_i^n \text{ in } X_i . 
\]
This means that there exist decreasing sequences $(g_i^n)$ in $X_i$ such that $\inf_n g_i^n = 0$ (in $X_i$) and $|h_i^n - h_i| \leq g_i^n$. Again, since the spaces $X_i$ are order ideals in $X_0+X_1$, 
\[
 h_i = o-\lim_{n\to\infty} h_i^n \text{ in } X_0+X_1 . 
\]
Since the order limit is additive by \cite[Proposition 1.1.11]{MN91}, 
\[
 h := h_0 + h_1 = o-\lim_{n\to\infty} (h_0^n +h_1^n) = o-\lim_{n\to\infty} h^n \text{ in } X_0+X_1 ,
\]
that is, there exists a decreasing sequence $(g^n)$ in $X_0+X_1$ such that $\inf_n g^n = 0$ (in $X_0+X_1$) and $|h^n-h| \leq g^n$. In particular, for every $n$,
\[
 0\leq (h^n-h)^+ \leq g^n .
\]
Since the sequence on the left is increasing, and since the sequence on the right is decreasing to $0$, necessarily $h_n\leq h$ for every $n$, that is, $h$ is an upper bound of the sequence $(h^n)$. From $0\leq h-h^n \leq g^n$ follows that $h$ is the least upper bound. We have thus proved that $\{ h_n : n\}$ admits a supremum in $X_0+X_1$, and hence $X_0+X_1$ is $\sigma$-Dedekind complete.
\end{proof}

\begin{remark} \label{rem.dedekind}
By \cite[Proposition 1.2.20, Proposition 1.1.8 (iv)]{MN91}, a Banach lattice is $\sigma$-Dedekind complete if and only if it has the principal projection property, that is, every principal band (a band generated by a single element) is a projection band. In the following theorem we really use the principal projection property of $X_0+X_1$, but we believe that $\sigma$-Dedekind completeness of a Banach lattice is a more common property. If $\vec{X} = (X_0,X_1)$ is an interpolation couple of compatible, $\sigma$-Dedekind complete Banach lattices, then by the previous lemma, $X_0+X_1$ is $\sigma$-Dedekind complete. The canonical, positive projection $P$ onto a principal band generated by an element $f\in X_0+X_1$ is on the positive cone necessarily given by
 \[
  P h = \sup \{ h\wedge n|f| : n\in\N \} \quad (h\in (X_0+X_1)^+) ,
 \]
the supremum being taken in $X_0+X_1$. Since $X_0$ and $X_1$ are $\sigma$-Dedekind complete, and if $h\in X_i^+$, then supremum in the representation of $Ph$ exists also in $X_i$ and coincides with $Ph$. In other words, the canonical projection onto a principal band in $X_0+X_1$ is always admissible, since its restrictions to $X_0$ and $X_1$ coincide with the canonical projections in these spaces. 
\end{remark}

\begin{theorem} \label{lem.sep-k-monotone.exact} 
Let $\vec{X} = (X_0,X_1)$ be an interpolation couple of compatible Banach lattices. Then:
\begin{itemize}
\item[(a)] Every exact partially $K$-monotone intermediate space of $\vec{X}$ is an exact interpolation space for positive, linear operators. 
\item[(b)] If 
\begin{itemize}
 \item[(1)] the couple $\vec{X}$ is an exact Calder\'on-Mityagin couple with respect to positive operators,
 \item[(2)] the spaces $X_0$ and $X_1$ are $\sigma$-Dedekind complete, and
 \item[(3)] for every pair of principal band projections $P$, $Q\in\cL^+ (\vec{X})$ and every pair of contractions $T_1\in\cL^+ (P\vec{X},Q\vec{X})$, $T_2\in\cL^+ ((I-P)\vec{X} ,(I-Q)\vec{X})$ the diagonal operator $T\in\cL (\vec{X})$ given by $T = QT_1P +(I-P)T_2(I-Q)$ is a contraction,
\end{itemize}
 then every exact interpolation space for positive, linear operators is exactly partially $K$-monotone.
\end{itemize}
\end{theorem}

\begin{proof}
 (a) Let $X$ be an exact partially $K$-monotone intermediate space, and let $S\in\cL^+ (\vec{X})$, $f\in X$. Then $0\leq (Sf)^+ \leq S(f^+)$ and therefore, by Lemma \ref{lem.basic} (c), for every $t\in (0,\infty )$,  
\begin{align*}
 K((Sf)^+,t) & = \inf \{ \| g_0\|_{X_0} + t \, \| g_1\|_{X_1} : g_i\in X_i \text{ and } (Sf)^+ = g_0 + g_1 \} \\
  & \leq \inf \{ \| g_0\|_{X_0} + t \, \| g_1\|_{X_1} : g_i\in X_i \text{ and } S (f^+) = g_0 + g_1 \} \\
  & \leq \inf \{ \| Sf_0 \|_{X_0} + t\, \| Sf_1\|_{X_1} : f_i\in X_i \text{ and } f^+= f_0 +f_1 \} \\
  & \leq \| S\|_{\cL (\vec{X})} \, \inf \{ \| f_0\|_{X_0} + t\, \| f_1\|_{X_1} : f_i \in X_i \text{ and } f^+ = f_0 +f_1\} \\
  & = \| S\|_{\cL (\vec{X})} \, K(f^+,t) .
\end{align*}
Similarly, one shows that $K((Sf)^-,t) \leq \| S\|_{\cL (\vec{X})} \, K(f^-,t)$. As a consequence, $Sf \ll \| S\|_{\cL (\vec{X})} \, f$.  Since $X$ is exactly partially $K$-monotone, this implies $Sf\in X$ and $\| Sf\|_X \leq  \| S\|_{\cL (\vec{X})} \, \| f\|_X$. Since $S$ and $f$ were arbitrary, this proves that $X$ is an exact interpolation space for positive, linear operators. 

(b) Assume now in addition that the interpolation couple $\vec{X}$ satisfies the conditions (1)--(3). Let $f\in X$ and $g\in X_0+X_1$ such that $g\ll_K f$, that is, $K(g^\pm , \cdot ) \leq K(f^\pm ,\cdot )$. Since $\vec{X}$ is an exact Calder\'on-Mityagin couple with respect to positive operators (assumption (1)), there exist admissible, positive, linear operators $T^+$, $T^-\in\cL^+ (\vec{X} )$ such that $\| T^\pm\|_{\cL (\vec{X})} \leq 1$ and $T^+f^+ = g^+$ and $T^-f^-=g^-$. Since $X_0$ and $X_1$ are $\sigma$-Dedekind complete (assumption (2)), and by Lemma \ref{lem.dedekind}, the space $X_0+X_1$ is $\sigma$-Dedekind complete, and therefore, by \cite[Proposition 1.2.11, p. 17]{MN91}, it has the principal projection property. By the proof of Lemma \ref{lem.dedekind} (see also Remark \ref{rem.dedekind}), the projections $P$ and $Q$ onto the bands generated by $f^+$ and $g^+$, respectively, are admissible, positive, linear contractions.  Note that $P f^+ = f^+$, $(I-P)f^- = f^-$, and similarly $Q g^+ = g^+$, $(I-Q)g^- = g^-$.
 We now set $T = Q T^+ P + (I-Q) T^- (I-P)$. Then $T$ is admissible, positive, linear and $Tf = Q^+ T^+ f^+ - Q^- T^-f^- = g^+ - g^- = g$. Moreover, by assumption (3), $T$ is a contraction. Since $X$ is an exact interpolation space for positive, linear operators, this implies $g\in X$ and $\| g\|_X \leq \| f\|_X$. Since $f$ and $g$ were arbitrary, this proves that $X$ is exactly partially $K$-monotone.
\end{proof}

\begin{remark} \label{rem.cm}
The assumption (3) from Theorem \ref{lem.sep-k-monotone.exact} is for example satisfied if there exist lattice norms $N_0$ and $N_1$ on $\R^2$ such that for every principal band projection $P\in\cL^+ (\vec{X})$ and for every $f\in X_i$ 
\[
 \| f\|_{X_i} = N_i (\| Pf\|_{X_i} , \| (I-P)f\|_{X_i} ) \quad (i=0,\, 1) .
\]
This condition is for example satisfied in $L^p$-spaces ($1\leq p\leq \infty$); simply choose $N_i$ as the $p$-norm on $\R^2$. \end{remark}

\begin{examples} \label{ex.calderon.mityagin}
(a) Let $(\Omega ,\mu)$ be a $\sigma$-finite measure space. Then, by the Calder\'on-Ryff theorem (\cite[Theorem 2.10, Corollary 2.11, pp. 114]{BeSh88}), the interpolation couple $(L^1 (\Omega ),L^\infty (\Omega ))$ is an exact Calder\'on-Mityagin couple with respect to positive operators. Moreover, the spaces $L^1$ and $L^\infty$ are $\sigma$-Dedekind complete. Finally, both spaces satisfy assumption (3) of Theorem \ref{lem.sep-k-monotone.exact} (b) (use Remark \ref{rem.cm} above). As a consequence, by Theorem \ref{lem.sep-k-monotone.exact}, the exact interpolation spaces for positive, linear operators and the exact partially $K$-monotone spaces coincide.

(b) Let $(\Omega ,\mu)$ be a measure space, let $w_0$, $w_1$ be two weights on $\Omega$, and let $p_0$, $p_1\in [1,\infty )$. By a result of Sparr \cite[Lemma 4.2]{Sp78} (see also Cwikel \cite[proof of Corollary 2, p. 234]{Cw76} in the case when $p_0\not= p_1$ or Sedaev \cite{Sd73} and Cwikel \cite[Theorem 4', p. 234]{Cw76} in the case $p_0 = p_1$), the interpolation couple $(L^{p_0}_{w_0} (\Omega ),L^{p_1}_{w_1} (\Omega ))$ is an exact interpolation Calder\'on-Mityagin couple with respect to positive operators; actually, for positive $g$, $f$ with $g\preceq_K f$, Lemma 4.2 in \cite{Sp78} only states that there exists an admissible, linear contraction $T$ such that $Tf=g$, but the proof shows that $T$ can be chosen to be positive. Moreover, the spaces $L^{p_0}_{w_0}$ and $L^{p_1}_{w_1}$ are $\sigma$-Dedekind complete. Finally, $L^p$-spaces satisfy assumption (3) of Theorem \ref{lem.sep-k-monotone.exact} (b) (use again Remark \ref{rem.cm}). As a consequence, by Theorem \ref{lem.sep-k-monotone.exact}, the exact interpolation 
spaces for positive, linear operators and the exact partially $K$-monotone spaces coincide; this observation is an analogue to \cite[Corollary 4.1]{Sp78}, which states that in this particular interpolation couple
the exact interpolation spaces for bounded, linear operators coincide with the exact $K$-monotone spaces.  

(c) For further examples of Calder\'on-Mityagin couples of spaces of measurable functions we refer to Cwikel an Keich \cite{CwKe01}: among others, they mention couples of Lorentz spaces \cite[Example 2.5]{CwKe01}, or the couple $(L^1,X)$, where $X$ is a rearrangement invariant Banach function space \cite[Section 5]{CwKe01}; note, however, that the Calder\'on-Mityagin property is hidden behind the property of exact monotonicity which is defined via optimal decomposability of functions with respect to the $K$-functional. Especially the fact that the mentioned interpolation couples are exact Calder\'on-Mityagin couples with respect to positive operators has to be checked. See also results by Astashkin \cite{As02} or Astashkin, Maligranda and Tikhomirov \cite{AsMaTi13}.
\end{examples}

\section{Interpolation of positive Gagliardo-Peetre operators, order preserving sublinear operators and order preserving Lipschitz operators}

Let $\vec{X}=(X_0,X_1)$ and $\vec{Y}=(Y_0,Y_1)$ be two interpolation couples of compatible Banach lattices. In this section we study interpolation of possibly nonlinear operators. We call an operator $S:X_0+X_1\supseteq\dom{S}\to Y_0+Y_1$, defined on a nonempty domain $\dom{S}\subseteq X_0+X_1$, {\em positive Gagliardo-Peetre operator} if
\begin{align*}
 & \exists C\geq 0 \, \forall f\in\dom{S} \, \forall f_i \in X_i \text{ with } f= f_0+f_1 \, \forall \varepsilon >0 \, \exists g_i\in Y_i \\
 & Sf = g_0+g_1 \text{ and } \| g^\pm_i \|_{Y_i} \leq C\, \| f^\pm_i \|_{X_i} + \varepsilon .
\end{align*}
The least possible constant $C\geq 0$ in the above definition is denoted by $\| S\|_{\GP^+}$. The set of all positive Gagliardo-Peetre operators from $\vec{X}$ into $\vec{Y}$ is denoted by $\GP^+ (\vec{X},\vec{Y})$.

The following lemma gives two characterisations and two necessary conditions for an operator to be a positive Gagliardo-Peetre operator. Assertion (b) is a slight adaptation of \cite[Proposition 4.1.3, p. 494]{BrKr91}; see also \cite[Proposition 2.1, p. 106]{BeSh88}.

\begin{lemma} \label{lem.properties.gpplus}
Let $\vec{X} = (X_0,X_1)$ and $\vec{Y} = (Y_0,Y_1)$ be two interpolation couples of compatible Banach lattices, and let $S:X_0+X_1\supseteq\dom{S}\to Y_0+Y_1$ be a (possibly nonlinear) operator. The following are true:
\begin{itemize}
\item[(a)] The operator $S$ is a positive Gagliardo-Peetre operator if and only if 
\begin{align*}
 & \exists C\geq 0 \, \forall f\in\dom{S}\, \forall f_i^\pm \in X_i^+ \text{ with } \, f^\pm= f_0^\pm+f_1^\pm \, \forall \varepsilon >0 \exists g_i^\pm\in Y_i^+ \\
 & (Sf)^\pm \leq g_0^\pm +g_1^\pm \text{ and } \| g^\pm_i \|_{Y_i} \leq C\, \| f^\pm_i \|_{X_i} + \varepsilon .
\end{align*}
The equivalence remains true if $(Sf)^\pm \leq \dots$ is replaced by $(Sf)^\pm = \dots$.
\item[(b)] The operator $S$ is a positive Gagliardo-Peetre operator with $\| S\|_{\GP^+} \leq C$ if and only if, for every $f\in\dom{S}$,
 \[
  Sf \ll_K C \, f . 
 \]
\item[(c)] If $S$ is a positive Gagliardo-Peetre operator, then, for every $f\in\dom{S}$, 
\begin{align*}
 f\geq 0 \quad & \Rightarrow \quad Sf\geq 0 , \quad \text{and} \\
 f\leq 0 \quad & \Rightarrow \quad Sf \leq 0 .
\end{align*}
In particular, if $0\in\dom{S}$, then $S0=0$.  
\item[(d)] If $S$ is a positive Gagliardo-Peetre operator, then, for $C=2\, \| S\|_{\GP^+}$, 
\begin{equation} \label{eq.gp}
\begin{split}
 & \forall f\in\dom{S} \, \forall f_i \in X_i \text{ with } f= f_0+f_1 \, \forall \varepsilon >0 \, \exists g_i\in Y_i \\
 & Sf = g_0+g_1 \text{ and } \| g_i \|_{Y_i} \leq C \, \| f_i \|_{X_i} + \varepsilon .
\end{split}
\end{equation}
\end{itemize}
\end{lemma}

\begin{proof}
\noindent (a) ``$\Rightarrow$'' Let $S$ be a positive Gagliardo-Peetre operator. Let $f\in\dom{S}$, and let $f_i \in X_i$ be such that $f=f_0+f_1$. By Lemma \ref{lem.basic} (a), that is, by replacing $f_i$ by $\hat{f}_i$, if necessary, we may without loss of generality assume that $f^\pm = f_0^\pm + f_1^\pm$. In this procedure, the norms $\| f_i^\pm\|_{X_i}$ become even smaller, that is, $\| \hat{f}_i^\pm \|_{X_i}\leq \| f_i^\pm\|_{X_i}$. For this choice of $f_i$ and for every $\varepsilon >0$ there exists $g_i\in Y_i$ such that $Sf = g_0+g_1$ and $\| g_i^\pm\|_{Y_i} \leq C\, \| f_i^\pm\|_{X_i} + \varepsilon$. Repeating the same procedure as above with the functions $g_i$, that is, by applying Lemma \ref{lem.basic} (a) again, we may assume that $(Sf)^\pm = g_0^\pm + g_1^\pm$. Since again the norms $\| g_i^\pm\|_{Y_i}$ become smaller in this procedure, the norm estimate between $g_i^\pm$ and $f_i^\pm$ remains true. We have thus proved the implication ``$\Rightarrow$''. 

``$\Leftarrow$'' In order to prove the converse implication, it suffices to note that if we find $g_i^\pm\in Y_i$ such that $(Sf)^\pm \leq g_0^\pm + g_1^\pm$, then we can also find $\tilde{g}_i^\pm\in Y_i$ such that $\| \tilde{g}_i^\pm\|_{Y_i} \leq \| g_i^\pm\|_{Y_i}$ and $(Sf)^\pm = g_0^\pm + g_1^\pm$. This is an exercise.

(b)  ``$\Rightarrow$'' Assume that $S\in \GP^+ (\vec{X},\vec{Y})$ with $\| S\|_{\GP^+}\leq C$. Let $f\in\dom{S}$ and choose $f_i\in X_i$ such that $f^\pm = f_0^\pm + f_1^\pm$ (compare with Lemma \ref{lem.basic} (a)). By definition of a positive Gagliardo-Peetre operator, for every $\varepsilon >0$ there exist $g_0\in Y_0$ and $g_1\in Y_1$ such that $Sf = g_0 + g_1$ and $\| g^\pm_i \|_{Y_i} \leq C\, \| f^\pm_i \|_{X_i} + \varepsilon$. Then $(Sf)^+ \leq g_0^+ +g_1^+$ and hence, for every $t\in (0,\infty )$,
 \begin{align*}
  K((Sf)^+ ,t) & \leq K(g_0^+ + g_1^+ ,t) \\
   & \leq \| g_0^+\|_{Y_0} + t\, \| g_1^+ \|_{Y_1} \\
   & \leq C\, (\| f_0^+\|_{X_0} + t\, \| f_1^+\|_{X_1} ) + (1+t)\, \varepsilon .
 \end{align*}
 Taking the infimum over all possible representations $f = f_0+f_1$ (satisfying $f^+ =  f_0^+ + f_1^+$) and over $\varepsilon >0$, this implies
 \[
  K((Sf)^+ , t) \leq C\, K(f^+ ,t) .
 \]
 Similarly, one proves $K((Sf)^- ,\cdot ) \leq C\, K(f^-,\cdot )$. Hence, $Sf \ll_K C \, f$.
 
 ``$\Leftarrow$'' For the converse, fix $f\in\dom{S}$. Following \cite[Corollary 3.1.29]{BrKr91}, we note that $Sf\ll_K C\, f$ is equivalent to $Sf\ll_{K_\infty} C\, f$ in the sense that, for every $t>0$,
\[
 K_\infty((Sf)^\pm ,t) \leq C \, K_\infty(f^\pm,t) ,
\]
where 
\[
K_\infty (f,t) := \inf \{ \max \{ \| f_0\|_{X_0} , t\, \|f_1\|_{X_1}\} : f_i\in X_i \text{ and } f=f_0 +f_1\} 
\]
is a variant of the $K$-functional. Fix $\varepsilon >0$ and $f^+_i\in X_i^+$ with $f^+=f_0^+ + f_1^+$. 

Assume first that $f_i^+\neq 0$. Then, for $t := \|f^+_0\|_{X_0}/\|f^+_1\|_{X_1} \in (0,\infty )$ there exist $g_i^+\in Y_i^+$ such that $(Sf)^+ =g_0^+ + g_1^+$ and 
\begin{align*}
   \max \{ \| g_0^+\|_{Y_0},  t\, \| g_1^+ \|_{Y_1} \} & \leq  C \, \max \{ \| f_0^+\|_{X_0}, t\, \| f_1^+ \|_{X_1} \} +\varepsilon \\
   & = C\, \| f_0^+\|_{X_0} + \varepsilon \\
   & = C\, t\, \| f_1^+\|_{X_1} + \varepsilon.
\end{align*}
Thus, 
\[
 \|g_0^+\|_{Y_0} \leq C\, \| f_0^+\|_{X_0} + \varepsilon \text{ and } \|g_1^+\|_{Y_1} \leq C\, \|f_1^+\|_{X_1} + \frac{\varepsilon}{t} .
\]
Assume next that $f_0^+=0$. Then, for every $t>0$ there exist $g_i^+\in Y_i^+$ such that $(Sf)^+ =g_0^+ + g_1^+$ and 
\[
  \max \{ \| g_0^+\|_{Y_0},  t\, \| g_1^+ \|_{Y_1} \} \leq  C t\, \| f_1^+ \|_{X_1} + \varepsilon .
\]
Therefore, 
\[
 \|g_0^+\|_{Y_0} \leq C\, t\, \| f_0^+\|_{X_0} + \varepsilon \text{ and } \|g_1^+\|_{Y_1} \leq C\, \|f_1^+\|_{X_1} + \frac{\varepsilon}{t} .
\]
Similarly one proceeds in the case when $f^+_1=0$. Since in the first case, $\varepsilon >0$ is arbitrary, and since in the second and third case, $t>0$ and $\varepsilon >0$ are arbitrary (in particular, one may choose $t\leq 1$ or $t\geq 1$), we conclude that for every $\varepsilon >0$ and every $f_i^+\in X_i^+$ with $f^+ = f_0^+ + f_1^+$ there exist $g_i\in Y_i^+$ such that $(Sf)^+ =g_0^+ + g_1^+$ and $\| g_i^+\|_{Y_i}\leq C\, \| f_i^+\|_{X_i} + \varepsilon$. Similarly one proceeds for estimating $(Sf)^-$. As a consequence, $S$ is a positive Gagliardo-Peetre operator.

(c) Let $f\in\dom{S}$ be such that $f\geq 0$. Then $f= f_0 +f_1$ for some $f_i \in X_i^+$. By the characterisation from (a), for every $\varepsilon >0$ there exists $g_i^-\in Y_i^+$ such that $0\leq (Sf)^- \leq g_0^- + g_1^-$ and $\| g_i^-\|_{Y_i} \leq C\, \| f_i^-\|_{X_i} + \varepsilon = \varepsilon$. In particular, $\| (Sf)^- \|_{Y_0+Y_1} \leq 2\varepsilon$. Since $\varepsilon >0$ is arbitrary, this yields $(Sf)^- = 0$, and hence $Sf\geq 0$. Similarly, one proceeds if $f\in\dom{S}$ and $f\leq 0$. 

(d) Let $f_i\in X_i$ such that $f=f_0+f_1\in\dom{S}$, and let $\varepsilon >0$. Let $g_i\in Y_i$ be as in the definition of a positive Gagliardo-Peetre operator. Then $Sf = g_0+g_1$ and 
\begin{align*}
 \| g_i\|_{Y_i} & \leq \| g_i^+\|_{Y_i} + \| g_i^-\|_{Y_i} \\
  & \leq \| S\|_{\GP^+} \, (\| f_i^+\|_{X_i} + \| f_i^-\|_{X_i} ) + 2\varepsilon \\
  & \leq 2\, \| S\|_{\GP^+} \, \| f_i\|_{X_i} + 2\varepsilon . 
\end{align*}
\end{proof}

Given two interpolation couples $\vec{X} = (X_0,X_1)$ and $\vec{Y} = (Y_0,Y_1)$ of Banach spaces we call an operator $S:X_0+X_1\supseteq\dom{S}\to Y_0+Y_1$ a {\em Gagliardo-Peetre operator} if it satisfies the condition \eqref{eq.gp} from Lemma \ref{lem.properties.gpplus} (d) for some constant $C\geq 0$. The least possible constant $C\geq 0$ in \eqref{eq.gp} is denoted by $\| S\|_{\GP}$. The set of all Gagliardo-Peetre operators from $\vec{X}$ into $\vec{Y}$ is denoted by $\GP (\vec{X}, \vec{Y})$. Compare our definition of Gagliardo-Peetre operators with \cite[Definition 4.1.1, p. 493]{BrKr91}, where Gagliardo-Peetre operators are necessarily defined on $\dom{S} = X_0+X_1$. We consider here also Gagliardo-Peetre operators on possibly smaller domains. By Lemma \ref{lem.properties.gpplus} (d), every positive Gagliardo-Peetre operator between interpolation couples of Banach lattices is a  Gagliardo-Peetre operator and $\| S\|_{\GP} \leq 2\, \| S\|_{\GP^+}$, but the converse is not true: in fact, every 
admissible, bounded, 
linear operator is a Gagliardo-Peetre operator (also in the sense of \cite{BrKr91}), but it is not necessarily positive as it should be by Lemma \ref{lem.properties.gpplus} (c). 

\begin{remark}
 Lemma \ref{lem.properties.gpplus} (c) is a justification for the notion of {\em positive} Gagliardo-Peetre operator. Every admissible, positive, linear operator is a positive Gagliardo-Peetre operator; compare also with Lemma \ref{lem.lipschitz.gp} below. Nevertheless, the notion of a positive operator might be ambiguous in the nonlinear setting. Every sublinear operator $S:X_0+X_1\supseteq \dom{S} \rightarrow X_0+X_1$ which is bounded in $X_i$ (in the sense that $\| Sf\|_{X_i} \leq C\, \| f\|_{X_i}$ for some constant $C\geq 0$ and every $f\in X_i$) is a classical Gagliardo-Peetre operator. However, even if it maps into the positive cone of $X_0+X_1$, it is not a positive Gagliardo-Peetre operator if $\dom{S}$ contains a negative vector. For example, the Hardy-Littlewood maximal operator $M$ is a classical Gagliardo-Peetre operator on the interpolation couple $(L^p (\R^N), L^\infty (\R^N))$, and $Mf\geq 0$ for every $f\in L^p+L^\infty$, but $M$ is not a positive Gagliardo-Peetre operator.
\end{remark}

\begin{lemma} \label{lem.properties.gp}
 Let $\vec{X} = (X_0,X_1)$ and $\vec{Y} = (Y_0,Y_1)$ be two interpolation couples of compatible Banach lattices, and let $S: X_0+X_1\supseteq \dom{S} \to Y_0+Y_1$ be an operator. Then: 
 \begin{itemize}
  \item[(a)] The operator $S$ is a Gagliardo-Peetre operator with $\| S\|_{\GP} \leq C$ if and only if
  \begin{align*}
 & \forall f\in \dom{S}\, \forall f_i \in X_i^+ \textrm{ with } |f|= f_0+f_1 \, \forall \varepsilon >0 \exists g_i\in Y_i^+ \\
 & |Sf| \leq g_0+g_1 \text{ and } \| g_i \|_{Y_i} \leq C\, \| f_i \|_{X_i} + \varepsilon \quad (i=0,1).
 \end{align*}
 \item[(b)] The operator $S$ is a Gagliardo-Peetre operator with $\| S\|_{\GP} \leq C$ if and only if, for every $f\in\dom{S}$, 
 \[
  Sf \preceq_K C \, f. 
 \]
 \item[(c)] If $S$ is a Gagliardo-Peetre operator such that $\dom{S}$ is a (not necessarily linear) lattice and $0\in \dom{S}$, and if $S$ is order preserving, 
 then $S$ is a positive Gagliardo-Peetre operator and $\| S\|_{\GP^+} \leq \| S\|_{\GP}$. 
 \item[(d)] If $S$ is a positive Gagliardo-Peetre operator and if $g\ll_K f$ implies $g\preceq_K f$ for every $g\in Y_0+Y_1$, $f\in X_0+X_1$, then $S$ is a Gagliardo-Peetre operator and $\| S\|_{\GP}\leq \|S\|_{\GP^+}$. In particular, if $\vec X$ is an exact Calder\'on-Mityagin couple with respect to positive operators, then each positive Gagliardo-Peetre operator is a Gagliardo-Peetre operator with $\| S\|_{\GP}\leq \|S\|_{\GP^+}$. 
 \end{itemize}
\end{lemma}

\begin{proof} 
Assertion (a) is straightforward, and assertion (b) is \cite[Proposition 4.1.3, p. 494]{BrKr91}; see also \cite[Proposition 2.1, p. 106]{BeSh88}. 

(c) Assume that $S$ is a Gagliardo-Peetre operator such that $\dom{S}$ is a lattice and $0\in\dom{S}$, and that $S$ is order preserving. From the definition of Gagliardo-Peetre operator follows $S0=0$ (compare with Lemma \ref{lem.properties.gpplus} (c)), and since $S$ is order preserving, we deduce  $Sf\geq 0$ (resp. $Sf\leq 0$) whenever $f\geq 0$ (resp. $f\leq 0$). Since $\dom{S}$ is a lattice and $0\in\dom{S}$, $f\in\dom{S}$ implies $f^+,-f^-\in\dom{S}$. Moreover, since $f\leq f^+$, since $S(f^+) \geq 0$, and since $S$ is order preserving, 
\[
(Sf)^+\leq S(f^+) \text{ for every } f\in\dom{S} .
\]
Similarly, since $-f^-\leq f$, since $S(-f^-) \leq 0$, and since $S$ is order preserving, 
\[
(Sf)^-\leq S(-f^-)^-= - S(-f^-) = |S(-f^-)| \text{ for every } f\in \dom{S} . 
\]
If $\varepsilon>0$ and $f\in \dom{S}$, $f^\pm_i\in X^+_i$ with $f^\pm = f^\pm_0 + f^\pm_1$, then the existence of vectors $g^+_i\in Y^+_i$ (respectively, $g^-_i\in Y^+_i$) as in the characterisation of positive Gagliardo-Peetre operators from Lemma \ref{lem.properties.gpplus} (a) follows from the assumption that $S$ is a Gagliardo-Peetre operator and from the characterisation of Gagliardo-Peetre operators from assertion (a) above applied to the functions $f^+\in\dom{S}$ (respectively, $-f^-\in\dom{S}$) and their decompositions $f^\pm = f_0^\pm + f_1^\pm$.  By Lemma \ref{lem.properties.gpplus} (a), $S$ is a positive Gagliardo-Peetre operator and $\|S\|_{\GP^+}\leq \|S\|_{\GP}$.

Assertion (d) is a straightforward consequence of the characterisations of (positive) Gagliardo-Peetre operators from Lemma \ref{lem.properties.gpplus} (b) and from assertion (b) above. 
\end{proof}

\begin{remark}\label{rem1}
 In the context of Lemma \ref{lem.properties.gp} (d) it should be noted that there are interpolation couples of compatible Banach lattices $\vec X$ which are not Calder\'on-Mityagin couples with respect to positive operators but the implication ``$g\ll_Kf$ implies $g\preceq_Kf$'' holds for every $f$, $g\in X_0+X_1$. An example is given by $\vec X := (L^1(\bT), \cC(\bT))$. It is known that $\vec X$ is not a Calder\'on-Mityagin couple, and therefore it is not a Calder\'on-Mityagin couple with respect to positive operators; see \cite{Cw76} and Remark \ref{rem2} below. On the other hand, the $K$-functionals with respect to the couple $(L^1(\bT), \cC(\bT))$ and the couple $(L^1(\bT), L^\infty(\bT)))$ coincide, and the implication ``$g\ll_Kf$ implies $g\preceq_Kf$'' is true in $(L^1(\bT),L^\infty(\bT))$, and thus also in $(L^1(\bT), \cC(\bT))$.
\end{remark}

In the following, we consider two classes of positive Gagliardo-Peetre operators, namely order preserving subadditive operators and order preserving Lipschitz operators. The computations in both cases are very similar. 

\begin{lemma} \label{lem.subadditive.gp}
Let $\vec{X} = (X_0,X_1)$ and $\vec{Y} = (Y_0,Y_1)$ be two interpolation couples of compatible Banach lattices, and let $S: X_0+X_1 \to Y_0+Y_1$ be order preserving, {\em subadditive} in the sense that 
\begin{align*}
 & S(f+g) \leq Sf + Sg  \text{ for every } f, \, g\in X_0+X_1 , \text{ or} \\
 & |S(f+g)| \leq |Sf| + |Sg|  \text{ for every } f, \, g\in X_0+X_1 ,
\end{align*}
and admissible in the sense that there exists a constant $C\geq 0$ such that, for every $f\in X_i$, 
\[
 Sf\in Y_i \text{ and } \| Sf\|_{Y_i} \leq C\, \| f\|_{X_i} .
\]
Then $S$ is a positive Gagliardo-Peetre operator and $\| S\|_{\GP^+} \leq C$. 
\end{lemma}

\begin{proof}
We assume that $S: X_0+X_1 \to Y_0+Y_1$ is admissible, order preserving, and subadditive in the sense that $S(f+g) \leq Sf + Sg$  for every $f$, $g\in X_0+X_1$. The arguments in the other definition of subadditivity are very similar. 

Let $f\in X_0+X_1$ and $f_i\in X_i$ be such that $f^\pm = f_0^\pm + f_1^\pm$. Then, since $S$ is order preserving and subadditive,
 \[
  0 \leq (Sf)^+ \leq S(f^+) \leq S(f_0^+) + S(f_1^+) =: g_0^+ + g_1^+ ,
 \]
with 
\begin{align*}
 \| g_0^+ \|_{X_0} & \leq C \, \| f_0^+\|_{X_0} \quad \text{and} \\
 \| g_1^+ \|_{X_1} & \leq C \, \| f_1^+\|_{X_1} .
\end{align*}
Similarly, the equality $-f_0^- = -f^- + f_1^-$ and subadditivity yield 
\[
 0 \leq (Sf)^- \leq -S(-f^-) \leq - S(-f_0^-) + S(f_1^-) =: g_0^- + g_1^- ,  
\]
with 
\begin{align*}
 \| g_0^- \|_{X_0} & \leq C \, \| f_0^-\|_{X_0} \quad \text{and} \\
 \| g_1^- \|_{X_1} & \leq C \, \| f_1^-\|_{X_1} .
\end{align*}
By Lemma \ref{lem.properties.gpplus} (a), this proves that $S$ is a positive Gagliardo-Peetre operator and $\| \tilde{S}\|_{\GP^+} \leq C$. 
\end{proof}

An {\em admissible Lipschitz operator} between two interpolation couples $\vec{X} = (X_0,X_1)$ and $\vec{Y} = (Y_0,Y_1)$ of Banach spaces is an operator $S:X_0+X_1\supseteq\dom{S}\to Y_0+Y_1$ for which there exists a constant $C\geq 0$ such that, for every $f$, $\hat{f}\in\dom{S}$, $i\in \{ 0,1\}$,
\[
 \| Sf-S\hat{f} \|_{Y_i} \leq C\, \| f-\hat{f}\|_{X_i} ; 
\]
\noindent here, we interpret the right-hand side as $\infty$ if $f-\hat{f}\not\in X_i$, and accordingly the left-hand side being finite means that $Sf-S\hat{f}\in Y_i$. The least possible constant $C\geq 0$ such that the above inequality holds is denoted by $\| S\|_\Lip$. The set of all admissible Lipschitz operators between $\vec{X}$ and $\vec{Y}$ is denoted by $\Lip  (\vec{X} , \vec{Y})$. Moreover, we set $\Lip_0 (\vec{X},\vec{Y}) := \{ S\in\Lip (\vec{X},\vec{Y}) : 0\in\dom{S} \text{ and } S0=0\}$, and we call the operators in $\Lip_0$ {\em normalised, admissible Lipschitz operators}. If $\vec{X}$ and $\vec{Y}$ are interpolation couples of compatible Banach lattices, we further denote by $\Lip^+ (\vec{X} ,\vec{Y} )$ (resp. $\Lip_0^+ (\vec{X},\vec{Y})$) the set of all (normalised) admissible, order preserving Lipschitz operators. Here, we call a nonlinear operator $S$ {\em order preserving} if $Sg\leq Sf$ whenever $g\leq f$.

\begin{remark}[Renormalisation] \label{rem.renormalisation}
Let $S: X_0 +X_1 \supseteq \dom{S}\to Y_0 + Y_1$ be any operator such that $\dom{S}\not=\emptyset$. Choosing $\hat{f}\in\dom{S}$ and defining $\tilde{S} : X_0+X_1\supseteq\dom{\tilde{S}}\to Y_0+Y_1$ by $\tilde{S} f = S (f+\hat{f}) - S\hat{f}$, we obtain an operator $\tilde{S}$ satisfying $0\in\dom{\tilde{S}}$ and $\tilde{S}0=0$. Moreover, $\tilde{S}$ is an admissible Lipschitz operator if and only if $S$ is an admissible Lipschitz operator. If $S$ is an admissible Lipschitz operator, then $\tilde{S} (X_0\cap\dom{S}) \subseteq X_0$ and $\tilde{S}(X_1\cap\dom{S})\subseteq X_1$ and, for every $f\in\dom{\tilde{S}}$, 
\begin{equation} \label{eq.lip.norm}
 \| \tilde{S} f\|_{X_0} \leq C\, \| f\|_{X_0} \text{ and } \| \tilde{S} f\|_{X_1}\leq C\, \| f\|_{X_1} 
\end{equation}
for $C= \| S\|_\Lip$. Conversely, if for every $\hat{f}\in\dom{S}$ the renormalised operator $\tilde{S}$ is admissible and satisfies \eqref{eq.lip.norm} (with a constant $C\geq 0$ independent of $\hat{f}$), then $S$ is an admissible Lipschitz operator and $\| S\|_\Lip\leq C$. By this renormalisation, it is therefore no loss of generality if we restrict in the following to operators $S$ satisfying $0\in\dom{S}$ and $S0=0$, at least in the proofs. We finally remark that if $\vec{X} = (X_0,X_1)$ and $\vec{Y}=(Y_0,Y_1)$ are interpolation couples of compatible Banach lattices, then $\tilde{S}$ is order preserving if and only if $S$ is order preserving. \\
\end{remark}

We call a subset $A$ of an ordered Banach space $X$ {\em solid} if $g\leq h\leq f$ for $f$, $g\in A$ and $h\in X$ implies $h\in A$. 

\begin{lemma} \label{lem.lipschitz.gp}
Let $\vec{X} = (X_0,X_1)$ and $\vec{Y} = (Y_0,Y_1)$ be two interpolation couples of compatible Banach lattices, and let $S\in\Lip^+ (\vec{X} ,\vec{Y})$ be an admissible, order preserving Lipschitz operator such that $\dom{S}$ is a solid lattice. Then, for every $f$, $\hat{f}\in\dom{S}$, 
\[
 Sf - S\hat{f} \ll_K \| S\|_\Lip \, (f-\hat{f}).
\]
Moreover, as a consequence,
\[
\cL^+(\vec X, \vec Y)\subseteq \{ S\in \Lip_0^+ (\vec{X} , \vec{Y}) : \dom{S} \text{ is a solid lattice}\} \subseteq \GP^+(\vec X, \vec Y) ,
\]
and, for every $S\in\Lip_0^+ (\vec{X} ,\vec{Y})$ such that $\dom{S}$ is a solid lattice, 
\[
\|S\|_{\GP^+}\leq \|S\|_{\Lip (\vec X, \vec Y)} . 
\]
\end{lemma}

\begin{proof}
Let $S\in\Lip^+ (\vec{X},\vec{Y})$ be such that $\dom{S}$ is a solid lattice. Fix $\hat{f}\in\dom{S}$, and consider the renormalized operator $\tilde{S}$ defined in Remark \ref{rem.renormalisation}. Then $0\in\dom{\tilde{S}}$ and $\tilde{S} 0 = 0$. Moreover, $\| \tilde{S}\|_{\Lip} = \| S\|_{\Lip}$, $\tilde{S}$ is order preserving and $\dom{\tilde{S}}$ is a solid lattice. By Lemma \ref{lem.properties.gpplus} (b), it suffices to show that $\tilde{S}$ is a positive Gagliardo-Peetre operator and $\|\tilde{S}\|_{\GP^+} \leq \| S\|_\Lip$. Fix $f\in\dom{\tilde{S}}$. Choose $f_i\in X_i$ such that $f^\pm = f_0^\pm + f_1^\pm$ (compare with Lemma \ref{lem.basic} (a)). From the lattice property of $\dom{\tilde{S}}$ and the property $0\in\dom{\tilde{S}}$ it follows that $f^+ = f\vee 0$, $-f^- = f\wedge 0 \in\dom{\tilde{S}}$. Then the assumption that $\dom{\tilde{S}}$ is solid and the inequality $0\leq f_i^+ \leq f^+$ together imply $f_i^+\in\dom{\tilde{S}}$. Hence, since $\tilde{S}$ is order preserving, 
\[
0\leq (\tilde{S}f)^+\leq \tilde{S}(f^+) = (\tilde{S}(f^+_0+f^+_1) - \tilde{S}(f^+_1)) + \tilde{S}(f^+_1)=:g^+_0 + g^+_1,
\]
and since $\tilde{S}$ is an admissible Lipschitz operator,
\begin{align*}
 \| g_0^+ \|_{X_0} & \leq \| S\|_\Lip \, \| f_0^+\|_{X_0} \quad \text{and} \\
 \| g_1^+ \|_{X_1} & \leq \| S\|_\Lip \, \| f_1^+\|_{X_1} .
\end{align*}
Similarly, $-f_i^-\in\dom{\tilde{S}}$ and 
\[
 0\geq -(\tilde{S}f)^- \geq \tilde{S} (-f^-) = (\tilde{S} (-f_0^- - f_1^-) -\tilde{S} (-f_1^-)) + \tilde{S} (-f_1^-) =: - g_0^- - g_1^- ,  
\]
with 
\begin{align*}
 \| g_0^- \|_{X_0} & \leq \| S\|_\Lip \, \| f_0^-\|_{X_0} \quad \text{and} \\
 \| g_1^- \|_{X_1} & \leq \| S\|_\Lip \, \| f_1^-\|_{X_1} .
\end{align*}
By Lemma \ref{lem.properties.gpplus} (a), this proves that $\tilde{S}$ is a positive Gagliardo-Peetre operator and $\| \tilde{S}\|_{\GP^+} \leq \| S\|_\Lip$. 

In particular, we have proved that any normalised, admissible, order preserving Lipschitz operator $S$, such that the domain $\dom{S}$ is a solid lattice, is a positive Gagliardo-Peetre operator and $\| S\|_{\GP^+} \leq \| S\|_\Lip$. Finally, it suffices to note that every admissible, positive, linear operator is Lipschitz continuous, order preserving, and its domain (being the whole space $X_0+X_1$) is a solid lattice.  
\end{proof}

\begin{theorem} \label{thm.gp}
 Let $\vec{X} = (X_0,X_1)$ be an interpolation couple of compatible Banach lattices, and let $X$ be an intermediate space of $\vec{X}$. Consider the following assertions:
\begin{itemize} 
\item[(i)] The space $X$ is an exact partially $K$-monotone space.

\item[(ii)] The space $X$ is an exact interpolation space for positive Gagliardo-Peetre operators on $\vec{X}$, in the sense that for every $S\in \GP^+ (\vec{X})$ and every $f\in X \cap \dom{S}$ one has $Sf\in X$ and $\| Sf\|_X \leq \| S\|_{\GP^+} \, \| f\|_X$. 
 
\item[(iii)] The space $X$ is an exact interpolation space for order preserving Lipschitz operators on $\vec{X}$, for which the domain is a solid lattice, that is, for every $S\in\Lip^+ (\vec{X})$ such that $\dom{S}$ is a solid lattice, and for every $f$, $\hat{f}\in\dom{S}$ one has $\| Sf-S\hat{f}\|_X \leq \| S\|_{\Lip} \|f-\hat{f}\|_X$.  

\item[(iv)] The space $X$ is an exact interpolation space for positive, linear operators on $\vec{X}$.  
\end{itemize}
Then (i)$\Leftrightarrow$(ii)$\Rightarrow$(iii)$\Rightarrow$(iv). If $\vec{X} = (X_0,X_1)$ is an exact Calder\'on-Mityagin couple with respect to positive operators, and if both $X_0$ and $X_1$ are $\sigma$-Dedekind complete and satisfy the condition (3) of Theorem \ref{lem.sep-k-monotone.exact} (b), then all four assertions are equivalent.
\end{theorem}

\begin{proof}
(i)$\Rightarrow$(ii) This implication follows directly from Lemma \ref{lem.properties.gpplus} (b) and the definition of exact partially $K$-monotone spaces. 

(ii)$\Rightarrow$(i) Assume that $X$ is an exact interpolation space for positive Gagliardo-Peetre operators on $\vec{X}$, and let $f\in X$ and $g\in X_0+X_1$ be such that $g\ll_K f$. Consider the operator $S: X_0+X_1 \to X_0+X_1$ given by $Sh := 0$ for $h\not= f$ and $Sf = g$. By Lemma \ref{lem.properties.gpplus} (b), $S$ is a positive Gagliardo-Peetre operator and $\| S\|_{\GP^+} \leq 1$. Since $X$ is an exact interpolation space for positive Gagliardo-Peetre operators on $\vec{X}$, this implies $g\in X$ and $\| g\|_X \leq \| f\|_X$. We have thus proved that $X$ is an exact partially $K$-monotone space.

The implication (i)$\Rightarrow$(iii) follows from Lemma \ref{lem.lipschitz.gp}, while the implication (iii)$\Rightarrow$(iv) follows from the simple observation that every admissible positive, linear operator is order preserving, Lipschitz continuous and everywhere defined on $X_0+X_1$ (hence, its domain is a solid lattice).  

If, in addition, $\vec{X} = (X_0,X_1)$ is an exact Calder\'on-Mityagin couple with respect to positive operators, and if both $X_0$ and $X_1$ are $\sigma$-Dedekind complete and satisfy the condition (3) from Theorem \ref{lem.sep-k-monotone.exact} (b), then the remaining implication (iv)$\Rightarrow$(i) follows from Theorem \ref{lem.sep-k-monotone.exact} (b), and thus all four assertions are equivalent. 
\end{proof}

\begin{remarks} \label{rem2}
(a) There are exact interpolation spaces for bounded, linear operators (in particular, they are exact interpolation spaces for positive, linear operators) which are not exact interpolation spaces for positive Gagliardo-Peetre operators. As a consequence, the implication (iv)$\Rightarrow$(ii) in the previous theorem (or, equivalently, the implication $(iv)\Rightarrow (i)$) is not true in general.

Indeed, consider the interpolation couple $\vec{X} = (X_0,X_1)$ of finite-dimensional Banach lattices $X_0 = X_1 = \R^3$, equipped with the Lorentz norm $\|\cdot \|_v$ and the supremum norm $\|\cdot \|_\infty$, respectively. Here,
\[
\| x\|_v := \sum_{k=1}^3 v_k x_k^* , 
\]
where $x^*$ is the decreasing rearrangement of $x$ and $v_1 \geq v_2 \geq v_3 \geq 0$. By \cite{SeSe71} (see also \cite[Theorem 2]{As02}), if $v_1 = v_2 = 1$ and $v_3 = 0$, then this interpolation couple is not a Calder\'on-Mityagin couple. Hence, there exists an exact interpolation space $X$ for bounded, linear operators which is $K$-monotone but not exactly $K$-monotone \cite[p. 29]{CwNiSc03}. The latter means that there exist $f$, $g\in X$ such that $g\preceq_K f$ and $\|f\|_X < \|g\|_X$. Now, $g\preceq_K f$ is equivalent to $|g|\ll_K |f|$. It is easy to see that in an interpolation couple of compatible Banach {\em function} lattices every exact interpolation space is a Banach lattice itself. In particular, for every $f\in X$,  $|f|\in X$ and $\|f\|_X = \|\,|f|\,\|_X$. In our situation this implies that $|g|\ll_K |f|$ and $\| \,|f| \,\|_X < \|\,|g|\,\|_X$. Therefore, the space $X$ is not exactly partially $K$-monotone. 

(b) There are exact interpolation spaces for everywhere defined Lipschitz operators which are not exact interpolation spaces for positive Gagliardo-Peetre operators, that is, a weaker variant of assertion (iii) (consider only admissible Lipschitz operators with $\dom{S} = X_0+X_1$) does in general not imply assertion (ii). 

Indeed, the first part of the proof of \cite[Theorem 2.5.23; pp. 234]{BrKr91} shows that each exact interpolation space for bounded, linear operators of a finite-dimensional regular interpolation couple is an exact interpolation space for Lipschitz operators. We recall that an interpolation couple $\vec{X} = (X_0,X_1)$ of Banach spaces is {\em regular} if the intersection $X_0\cap X_1$ is dense both in $X_0$ and $X_1$. In the case of finite-dimensional spaces this just means $\dim X_0 = \dim X_1$. Hence, the example from point (a) above serves also as an example in this case of Lipschitz operators.

(c) We do not know  if there are exact interpolation spaces for positive operators which are not exact interpolation spaces for order preserving Lipschitz operators, that is, more precisely, whether the implication (iv)$\Rightarrow$ (iii) is true in general or not.   

(d) As mentioned in the introduction, there are very few results on the interpolation of positive / order preserving operators on Banach lattices. In the case of the special interpolation couple $(L^1,L^\infty )$ (compare with Theorem \ref{thm.main.l1.linfty}), a variant of the equivalence (i)$\Leftrightarrow$(iii) (with a different condition on the domain of the Lipschitz operator $S$, and without the assumption that the measure space is $\sigma$-finite) is contained in B\'enilan and Crandall \cite{BeCr91}. It is basically a consequence of the Brezis-Strauss variant of the Riesz interpolation theorem \cite{BrSt73}; see also the discussion after Corollary \ref{cor.bc1} below. 

(e) The implication (i)$\Rightarrow$(ii) of Theorem \ref{thm.gp} together with Lemma \ref{lem.subadditive.gp} yields a result by Maligranda \cite[Theorem 6]{Mg89b} on interpolation of order preserving, subadditive operators in Banach lattices of measurable functions. In general Banach lattices, this result has been proved by Masty{\l}o \cite[Theorem 2.1]{Mas12}, however, under the stronger condition that the operators in question are positive, order preserving and {\em sublinear}. Here, sublinear means subadditive in the sense that $S(f+g) \leq Sf + Sg$, and in addition $S(\lambda f) = \lambda \, Sf$ for every $f\in X_0+X_1$ and every $\lambda >0$. The latter homogeneity condition can actually be dropped.
\end{remarks}

\section{The interpolation couple $(L^1,L^\infty )$} 

Let $(\Omega ,\mu )$ be a measure space and consider the special interpolation couple $(L^1,L^\infty ) := (L^1 (\Omega ),L^\infty (\Omega ))$. Let ${\mathcal M} = {\mathcal M} (\Omega )$ be the set of the measurable, complex valued functions on $\Omega$, and let ${\mathcal M}^+ = {\mathcal M} (\Omega )^+$ be the cone of real valued, non-negative functions. Denote by $f^*$ the decreasing rearrangement of a measurable function, and by $f^{**}$ the second rearrangement, which is given by $f^{**}(t)  = \frac{1}{t} \int_0^t f^* (s) \; ds$; see \cite[Definition 3.1, p. 52]{BeSh88}. Using the notion of second rearrangement, we introduce the {\em Hardy-Littlewood-Polya relation} $\preceq$ and the {\em B\'enilan-Crandall relation} $\ll$ on $\cM$ by setting, for every $f$, $g\in \cM$,
\begin{align*}
 g\preceq f \quad & :\Leftrightarrow \quad g^{**} \leq f^{**} \text{ pointwise everywhere, and} \\
 g \ll f \quad & : \Leftrightarrow \quad g^+\preceq f^+ \text{ and } g^-\preceq f^- .
\end{align*}
The relation $\ll$ was introduced in B\'enilan \& Crandall \cite{BeCr91}, which is the reason why we call it B\'enilan-Crandall relation. 

An intermediate space $X$ of $(L^1,L^\infty )$ is {\em monotone} if there exists a constant $C\geq 1$ such that, for every $f\in X$ and every $g\in 
L^1+L^\infty$,
 \[
  g\preceq f \quad \Rightarrow \quad g\in X \text{ and } \| g\|_X \leq C\, \| f\|_X ,
 \]
and we call it {\em exactly monotone} if the above implication holds true with $C=1$. Similarly, we say that the space $X$ is {\em partially monotone} if there exists a constant $C\geq 1$ such that, for every $f\in X$ and $g\in L^1 +L^\infty$, 
 \[
  g\ll f \quad \Rightarrow \quad g\in X \text{ and } \| g\|_X \leq C\, \| f\|_X , 
 \] 
and similarly we say that it is {\em exactly partially monotone} if the implication holds true with $C=1$. Exact partially monotone spaces are actually called {\em normal spaces} by B\'enilan and Crandall, but we try to avoid the term ``normal''.\\ 

With these definitions at hand, we are finally in the position to prove Theorem \ref{thm.main.l1.linfty} from the Introduction.

\begin{proof}[Proof of Theorem \ref{thm.main.l1.linfty}]
The equivalence of assertions (ii)--(v) follows from Theorem \ref{thm.gp} in combination with Example \ref{ex.calderon.mityagin} (a). 

If the measure space $(\Omega ,\mu )$ is $\sigma$-finite, as we assume, then, by \cite[Theorem 6.2, p. 74]{BeSh88}, for every $f\in L^1+L^\infty$ and every $t\in (0,\infty )$,  
\[
 f^{**} (t) = \frac{1}{t} K(f,t) .
\]
Hence, in the case of $\sigma$-finite measure spaces, the relations $\preceq$ and $\preceq_K$ coincide on $L^1+L^\infty$.  Similarly, the relations $\ll$ and $\ll_K$ coincide, and hence the class of exact partially monotone spaces coincides with the class of exact partially $K$-monotone spaces. Hence, assertions (i) and (ii) are equivalent.
\end{proof}

\section{Applications}

\subsection{Nonlinear semigroups generated by completely accretive operators}

An operator $A\subseteq X\times X$ on a Banach space $X$ is {\em accretive} if, for every $(u,f)$, $(\hat{u},\hat{f})\in A$ and every $\lambda >0$,
\[
 \| u-\hat{u} + \lambda (f-\hat{f})\|_X \geq \| u - \hat{u}\|_X ,
\]
and it is {\em $m$-accretive} if it is accretive and in addition ${\rm range}\, (I+\lambda A) = X$ for some/all $\lambda >0$. 
By the Crandall-Liggett theorem \cite{CrLi71} (see also \cite[Theorem 4.3, p. 131]{Ba10}), if $A$ is $m$-accretive, then $-A$ generates a nonlinear, strongly continuous contraction semigroup $S = (S_t)_{t\geq 0}$ on $\overline{\dom{A}}$ in the sense that the Cauchy problem 
\[
 \dot u + Au \ni 0 \text{ in } (0,\infty ) , \quad u(0)=u_0 ,
\]
is wellposed, that is, for every $u_0\in\overline{\dom{A}}$ this Cauchy problem admits a unique mild solution. In this case, the orbits of the semigroup $u (t):= S_t u_0$ are the unique mild solutions. 

Now, let $(\Omega ,\mu )$ be a $\sigma$-finite measure space, and assume that $X$ is an intermediate space of the interpolation couple $(L^1 (\Omega ) ,L^\infty (\Omega ))$. We call the semigroup {\em order preserving} if, for every $u$, $v\in X$,
\[
 u\leq v \quad \Rightarrow \quad S_t u \leq S_t v \text{ for every } t\in (0,\infty ), 
\]
and we call it {\em $L^\infty$-contractive} if, for every $u$, $v\in X$,
\[
 \| S_t u - S_t v\|_{L^\infty} \leq \| u-v\|_{L^\infty} \text{ for every } t\in (0,\infty ) ;
\]
as before, we interpret the right-hand side of this inequality as $\infty$ if $u-v\not\in L^\infty (\Omega )$, and the left-hand side being finite means $S_t u -S_tv\in L^\infty (\Omega )$. In a similar way we define {\em $L^1$-contractivity} of the semigroup: simply replace the $L^\infty$-norm by the $L^1$-norm in the above inequality. An operator $A\subseteq X\times X$ is {\em completely $m$-accretive} if it is $m$-accretive and if the semigroup $S$ generated by $-A$ is order preserving, $L^1$-contractive and $L^\infty$-contractive. From our abstract results, we obtain the following corollary. 

\begin{corollary} \label{cor.bc1}
Let $(\Omega ,\mu )$ be a $\sigma$-finite measure space, and let $X$ be an exact partially monotone interpolation space of the interpolation couple $(L^1 (\Omega ) ,L^\infty (\Omega ))$. Let $A\subseteq X\times X$ be a completely $m$-accretive operator, and let $S$ be the semigroup generated by $-A$. Assume that $\overline{\dom{A}}$ is a (not necessarily linear) solid lattice in $L^1 (\Omega ) + L^\infty (\Omega )$. Then $S$ is $X$-contractive for every exact partially monotone space $X$ in the sense that, for every $u$, $v\in X$,
\[
 \| S_t u - S_t v\|_{X} \leq \| u-v\|_{X} \text{ for every } t\in (0,\infty ).
\] 
\end{corollary}

It should be noted that Corollary \ref{cor.bc1} is a special case of \cite[Proposition 4.1]{BeCr91} which states that under the assumptions of Corollary \ref{cor.bc1} the semigroup $S$ is {\em completely contractive}. This means that $S$ is not only $X$-contractive for every exact partially monotone interpolation space, but that it is {\em $N$-contractive} for every so-called {\em normal} function $N : L^1 (\Omega ) + L^\infty (\Omega ) \to [0,\infty ]$. The class of normal functions contains all norms $\|\cdot\|_X$ (extended by $\infty$ outside $X$) of exact partially monotone interpolation spaces $X$, but it also contains functions which are not coming from norms. The notion of $N$-contractivity is, however, defined similarly as $X$-contractivity: simply replace the norm in $X$ by the function $N$. The result in Benilan and Crandall is basically a consequence of the Brezis-Strauss variant of Riesz' interpolation theorem \cite[Lemma 3 and Lemma 3$^*$]{BrSt73} (see also \cite[Proposition 2.1]{BeCr91}) which 
admits a rather elementary proof. We recall, however, that our aim was to provide a structural proof leading to interpolation results in general Banach lattices. 

\subsection{Nonlinear semigroups generated by Dirichlet forms}

Let $(\Omega ,\mu )$ be a $\sigma$-finite measure space. The subgradient of a convex, lower semicontinuous, proper function $\E : L^2 (\Omega ) \to \R \cup \{ +\infty\}$, given by 
\begin{align*}
 \partial\E := \{ (u,f) \in L^2 (\Omega ) \times L^2 (\Omega ) : \,\, & u\in\dom{\E} \text{ and for every } v\in L^2 (\Omega ) \\
  & \E (u+v) -\E(u) \geq \langle f , v\rangle_{L^2} \}
\end{align*}
is an $m$-accretive (equivalently, {\em maximal monotone}) operator on $L^2 (\Omega )$ \cite[Exemple 2.3.4]{Br73}, \cite[Theorem 2.8, p. 47]{Ba10}. Here, $\dom{\E} := \{ \E <+\infty\}$ is the {\em effective domain} of $\E$. The negative subgradient therefore generates a semigroup $S = (S_t)_{t\geq 0}$ of nonlinear contractions on $L^2 (\Omega )$ which is strongly continuous on $(0,\infty )$; note carefully that in contrast to the Banach space case, the semigroup is here defined everywhere on $L^2 (\Omega )$, but only for $u\in \overline{\dom{\E}}$ the orbit $t\mapsto S_t u$ is continuous up to $t=0$. 

A convex, lower semicontinuous, proper function $\E$ on $L^2 (\Omega )$ is called {\em Dirichlet form}, if the semigroup $S$ generated by its negative subgradient is order preserving and $L^\infty$-contractive. Barth\'elemy \cite{By96} and Cipriano \& Grillo \cite{CiGr03} have characterized Dirichlet forms intrinsically. In fact, the semigroup $S$ generated by the subgradient of $\E$ is order preserving if and only if, for every $u$, $v\in L^2 (\Omega )$, 
\[
 \E (u\wedge v) + \E (u\vee v) \leq \E (u) + \E (v) ,
\]
and it is $L^\infty$-contractive if and only if, for every $u$, $v\in L^2 (\Omega )$ and for every $\alpha\geq 0$, 
\[
 \E (u-\frac12 ((u-v+\alpha)^+ - (u-v-\alpha)^-)) + \E (v+\frac12 ((u-v+\alpha)^+ - (u-v-\alpha)^-)) \leq \E (u) + \E (v) .
\]
By a duality argument due to B\'enilan \& Picard \cite{BePi79a}, if the semigroup $S$ is $L^\infty$-contractive, then it is also $L^1$-contractive. The following result then follows immediately from Corollary \ref{cor.bc1}. 

\begin{corollary} \label{cor.bc2}
 Let $(\Omega ,\mu )$ be a $\sigma$-finite measure space. Let $\E$ be a Dirichlet form on $L^2 (\Omega )$, and let $S$ be the semigroup generated by its negative subgradient. Then $S$ is $X$-contractive for every exact partially monotone space $X$ in the sense that, for every $u$, $v\in L^2 (\Omega )$,
\[
 \| S_t u - S_t v\|_X \leq \| u-v\|_X \text{ for every } t\in (0,\infty ).
\] 
\end{corollary}

A typical example of a Dirichlet form is the energy of the $p$-Laplace operator, say with Neumann boundary conditions. In this case, $\Omega$ is an open subset in $\R^N$, $p\in (1,\infty )$, and, for every $u\in L^2 (\Omega )$,
\[
 \E (u) = \begin{cases}
           \frac{1}{p} \int_\Omega |\nabla u|^p & \text{if } \nabla u\in L^p (\omega ) , \\[2mm]
           \infty & \text{else,}
          \end{cases}
\]
Here, the gradient $\nabla u$ is to be understood in the distributional sense. One easily checks that $\E$ is convex, lower semicontinuous and that the effective domain $\dom{\E}$ is dense in $L^2 (\Omega )$, so that the semigroup $S$ generated by $-\partial\E$ ($=\Delta_p$, the $p$-Laplace operator with Neumann boundary conditions) is strongly continuous up to $t=0$ on the entire space $L^2 (\Omega )$. The fact that $\E$ is a Dirichlet form can be easily checked with the help of the Beurling-Deny type conditions due to Barth\'elemy and Cipriano \& Grillo. The fact that the semigroup $S$ is $X$-contractive for every exact partially monotone interpolation space $X$ of $(L^1(\Omega ),L^\infty (\Omega ))$ (Corollary \ref{cor.bc2}) then means that the associated parabolic problem
\begin{align*}
 \partial_t u - \Delta_p u = 0 & \text{ in } (0,\infty )\times\Omega , \\
 \partial_\nu u = 0 & \text{ in } (0,\infty ) \times\partial\Omega ,
\end{align*}
is wellposed in every exact partially monotone interpolation space and that it generates a contraction semigroup in these spaces.

\def\cprime{$'$} \def\cprime{$'$} \def\cprime{$'$} \def\cprime{$'$}
  \def\cprime{$'$} \def\cprime{$'$} \def\cprime{$'$} \def\cprime{$'$}
  \def\cprime{$'$} \def\cprime{$'$} \def\cprime{$'$} \def\cprime{$'$}
  \def\cprime{$'$} \def\cprime{$'$} \def\cprime{$'$} \def\cprime{$'$}
  \def\cprime{$'$} \def\cprime{$'$} \def\cprime{$'$} \def\cprime{$'$}
  \def\cprime{$'$} \def\cprime{$'$} \def\cprime{$'$} \def\cprime{$'$}
  \def\cprime{$'$} \def\cprime{$'$} \def\cprime{$'$} \def\cprime{$'$}
  \def\cprime{$'$} \def\cprime{$'$} \def\cprime{$'$}
  \def\ocirc#1{\ifmmode\setbox0=\hbox{$#1$}\dimen0=\ht0 \advance\dimen0
  by1pt\rlap{\hbox to\wd0{\hss\raise\dimen0
  \hbox{\hskip.2em$\scriptscriptstyle\circ$}\hss}}#1\else {\accent"17 #1}\fi}
  \def\cprime{$'$} \def\cprime{$'$} \def\cprime{$'$}
\providecommand{\bysame}{\leavevmode\hbox to3em{\hrulefill}\thinspace}
\providecommand{\MR}{\relax\ifhmode\unskip\space\fi MR }
\providecommand{\MRhref}[2]{%
  \href{http://www.ams.org/mathscinet-getitem?mr=#1}{#2}
}
\providecommand{\href}[2]{#2}

\end{document}